\documentclass{amsart}
\usepackage{amsthm, amsmath, amssymb, color}
\usepackage{hyperref}

\newtheorem{theorem}{Theorem}[section]
\newtheorem{lemma}[theorem]{Lemma}

\theoremstyle{definition}

\theoremstyle{remark}
\newtheorem{remark}[theorem]{Remark}

\numberwithin{equation}{section}

\begin{document}

\title[Waring's Problem for Polynomials]{Waring's Problem for Polynomial Rings and the Digit Sum of Exponents}

\author{Seth Dutter}
\address{}
\curraddr{}
\email{}

\author{Cole Love}
\address{}
\curraddr{}
\email{}

\thanks{Research supported in part by the UW--Stout Foundation.}

\date{\today}



\begin{abstract}
Let $F$ be an algebraically closed field of characteristic $p > 0$. In this paper we develop methods to represent arbitrary elements of $F[t]$ as sums of perfect $k$-th powers for any $k\in\mathbb{N}$ relatively prime to $p$. Using these methods we establish bounds on the necessary number of $k$-th powers in terms of the sum of the digits of $k$ in its base-$p$ expansion. As one particular application we prove that for any fixed prime $p > 2$ and any $\epsilon > 0$ the number of $(p^r-1)$-th powers required is $\mathcal{O}\left(r^{(2+\epsilon)\ln(p)}\right)$ as a function of $r$.
\end{abstract}

\maketitle

\section{Introduction}\label{section.intro}
Waring's problem asks whether there exists some positive integer $s$ such that every natural number can be represented as the sum of at most $s$ $k$-th powers. Vaughan and Wooley \cite{vaughan2002waring} provide a comprehensive survey of the classic version of this problem. Rather than focus on natural numbers, we consider the problem of representing an arbitrary polynomial as a sum of $k$-th powers over a field of positive characteristic. This question and its variants have been extensively studied. Paley \cite{paley1933theorems} was the first to establish upper bounds on the necessary size of $s$ in the case of polynomials over finite fields. More recently, Vaserstein \cite{vaserstein1992ramsey} and Liu and Wooley \cite{liu2007unrestricted} have improved upon these bounds. In the characteristic $0$ case Newman and Slater \cite{newman1979waring} have proven lower and upper bounds on the number of $k$-th powers needed to represent an arbitrary polynomial.

The characteristic of the field plays a crucial role in determining not only the number of $k$-th powers necessary, but whether the problem even has a solution. For example, over a field of characteristic $2$ it is never possible to write the monomial $t$ as a sum of perfect squares due to the Frobenius endomorphism. On the other hand, the upper bounds established in \cite{vaserstein1992ramsey} for the positive characteristic case are often lower than the proven lower bound in \cite{newman1979waring} for the characteristic $0$ case. Even with a fixed positive characteristic, the size of the field plays an important role, as algebraic extensions may be necessary for certain constructions.

Beyond the choice of the underlying field there is also the question of whether any restrictions are placed on the polynomials themselves. A particular representation
\[
f(t) = y_1(t)^k+\cdots+y_s(t)^k
\]
is said to be \emph{strict} if $\deg(y_i^k) < \deg(f)+k$ for all $i\in\{1,\ldots, s\}$. See \cite{liu2010waring} for a thorough treatment of the strict case. Here we will investigate the \emph{unrestricted} variant of Waring's problem wherein no condition is placed upon the degree of the $y_i$. A consequence of studying the unrestricted variant is that it suffices to represent $t$ as a sum of $k$-th powers. Indeed, if
\[
t = y_1(t)^k+\cdots+y_s(t)^k,
\]
then by function composition
\[
f(t) = y_1(f(t))^k+\cdots+y_s(f(t))^k
\]
for any polynomial $f(t)$. As we are only handling the unrestricted variant in this paper, the proofs will focus on representing $t$ as a sum of $k$-th powers.

Throughout this paper we will let $F$ be an algebraically closed field of characteristic $p > 0$. If $q$ is a power of $p$ we will define $\mathbb{F}_q$ to be the finite field with $q$ elements and we will naturally identify $\mathbb{F}_q$ with the subfield of $F$ having $q$ elements. As in \cite{voloch2016planar} we let $v(p,k)$ denote the smallest natural number $s$ such that every polynomial in $F[t]$ can be written as a sum of at most $s$ $k$-th powers. If no such $s$ exists, then we say $v(p,k) = \infty$.

We will frequently refer to the base-$q$ expansion of $k$. When doing so we will write
\[
k = k_1 q^{a_1}+\ldots +k_Nq^{a_N}
\]
where the $a_i$ are a strictly increasing sequence of nonnegative integers and the $k_i$ are the unique integers satisfying $0\leq k_i < q$. The $q^{a_1}$ factor will often be omitted because in all theorems $k$ will be assumed relatively prime to $q$ and therefore $a_1 = 0$. The $k$ studied in this paper will usually be sparse in the sense that most digits are $0$ and for this reason we will only list the nonzero $k_i$ in the expansion.

If $k=  k_1 p^{a_1}+\ldots +k_Np^{a_N}$ is the base-$p$ expansion of $k$, then Theorem 2 of Vaserstein \cite{vaserstein1992ramsey} establishes the bound
\begin{equation}\label{eq:vaserbound}
v(p, k) \leq \prod_{i=1}^N (k_i+1) - 1
\end{equation}
for all $k$ relatively prime to $p$. Our results are directed at improving this bound, under certain conditions, by relating $v(p, k)$ to the sum of the digits in the base-$q$ expansion of $k$ for some $q$. As such, we introduce the notation
\[
\gamma_q(k) = k_1+\cdots+k_N
\]
for this sum of digits. As a final piece of notation, we will need to make use of the indicator function
\[
	\mathbf{1}_M(x) = \left\{
	\begin{array}{ll}
	1	 & \text{if $x = M$}\\
	0	 & \text{if $x \neq M$}
	\end{array}
	\right..
\]

Before introducing our main theorem we note two additional reductions to Waring's problem. Firstly, if $k | k'$, then $v(p, k) \leq v(p, k')$. This follows from the observation that for any sum of $k'$-th powers we have
\[
y_1(t)^{k'}+\cdots+y_s(t)^{k'} = \left(y_1(t)^{k'/k}\right)^k+\cdots+\left(y_s(t)^{k'/k}\right)^k.
\]
Therefore, we will often establish a bound for some $v(p,k')$ from which a bound for $v(p,k)$ follows. This reduction is only possible because we are investigating the unrestricted variant of Waring's problem. Secondly, since we are working over an algebraically closed field, all constants are $k$-th powers. Thus, it suffices to represent $t$ as a linear combination of $k$-th powers of polynomials.

We now state our main theorem, the proof of which is provided in Section \ref{sect:mainthm}.

\begin{theorem}\label{thm:main}
	Let $p$ be prime and $k\in\mathbb{N}$ be relatively prime to $p$. If $M$, $n\in\mathbb{N}$ satisfy the conditions $\gamma_{p^n}(k)\leq M$ and $M-1\mid p^n-1$, then
	\[
	   v(p,k) \leq 2(M-1) + \mathbf{1}_M(\gamma_{p^n}(k)).
	\]
\end{theorem}

\begin{remark}\label{rem:car}
If $k \mid p^r+1$ for some natural number $r$, then we can take $M=2$ and $n=1$ in Theorem \ref{thm:main} to get $v(p, k) \leq 3$. This particular case also follows from Equation \ref{eq:vaserbound} and has been additionally studied by Car \cite{car2010sums,car2012sums} and Voloch \cite{voloch2016planar}.
\end{remark}

It should be noted that while we are working over an algebraically closed field all constructions will take place over some finite subfield of effectively computable size. In most instances the size of this subfield will be much larger than those needed for the constructions in  \cite{car2010sums}, \cite{car2012sums}, \cite{vaserstein1992ramsey}, and \cite{voloch2016planar}. However, when applicable in base-$p$, the bound established in Theorem \ref{thm:main} frequently gives an improvement over the bound in Equation \ref{eq:vaserbound}. While only a small proportion of numbers meet the hypothesis in base-$p$, in practice this limitation can often be worked around by using a larger base. In Section \ref{section.applications} we give examples of how this can be done, the results of which now follow.

\begin{theorem}\label{thm:odd}
Let $p$ be an odd prime, $k \in\mathbb{N}$ be relatively prime to $p$, and $p$ have odd order in $(\mathbb{Z}/k\mathbb{Z})^\times$. If $c$ is the least nonnegative integer such that $\gamma_p(k) \leq 2^c+1$, then
\[
    v(p, k) \leq 2^{c+1}+\mathbf{1}_{2^{c}+1}(\gamma_p(k)) < 4\gamma_p(k).
\]
\end{theorem}

\begin{remark}
If $p$ and $\ell$ are distinct primes with $p$ odd, then either $p$ has odd order in $(\mathbb{Z}/\ell\mathbb{Z})^\times$ or there exists some natural number $r$ such that $p^r\equiv -1\pmod{\ell}$. In the latter case, by Remark \ref{rem:car}, we have $v(p,\ell) \leq 3$. In the former case, by Theorem \ref{thm:odd}, we have $v(p, \ell) < 4\gamma_p(\ell)$. In either case, the bound is small when the exponent is prime.
\end{remark}

In the event that $p$ does not have odd order, or that $p=2$, bounds still follow from Theorem \ref{thm:main}, but they are no longer linear in $\gamma_p(k)$.

\begin{theorem}\label{thm:order}
Let $p$ be prime, $k\in\mathbb{N}$ be relatively prime to $p$, and $r$ the order of $p$ in $(\mathbb{Z}/k\mathbb{Z})^\times$. If $u\in\mathbb{N}$ is relatively prime to $r$ and $\gamma_p(k) \leq p^u$, then
\begin{equation}\label{eq:order}
    v(p,k) \leq 2(p^u-1) + \mathbf{1}_M(\gamma_p(k)).
\end{equation}
\end{theorem}

\begin{remark}
In the notation of the previous theorem, if $\gamma_p(k) > p$ and $\ell$ is the least prime not dividing $r$, then
\[
v(p, k) < 2\gamma_p(k)^\ell-1.
\]
Indeed, there must be some nonnegative integer $c$ such that $p^{\ell^c} < \gamma_p(k) \leq p^{\ell^{c+1}}$, so $p^{\ell^{c+1}} < \gamma_p(k)^\ell$. Let $u = \ell^{c+1}$ and the result follows by Equation \ref{eq:order}.
\end{remark}

A consequence of the bound in Equation \ref{eq:vaserbound} is that $v(p, k) < k$ for all $k\in\mathbb{N}$ satisfying the following conditions: $k$ is relatively prime to $p$, $k > p$, and $k$ is not of the form $p^r-1$ for any $r\in\mathbb{N}$. In the case that $k=p^r-1$ the established bound is $v(p, p^r-1) \leq p^r-1$. If $p$ and $r$ are odd, Theorem \ref{thm:odd} gives the generally sharper bound of $v(p, p^r-1) < 4(p-1)r$. However, this theorem does not apply if either $p$ or $r$ are even. In the following theorem we establish that $v(p,p^r-1)$ is bounded by a polynomial in $r$ of degree determined by $p$. The proof of this result appears in Section \ref{section.applications} and relies on bounds for the first Chebyshev function.

\begin{theorem}\label{thm:asymptotic}
Let $p$ be a fixed prime and $\epsilon > 0$ be a real number. Then
    \[
        v(p, p^r-1) = \left\{
            \begin{array}{ll}
                \mathcal{O}\left(r^{(2+\epsilon)\ln(p)}\right) & \text{if $p > 2$}\\
                \mathcal{O}\left(r^{2+\epsilon}\right) & \text{if $p=2$}
            \end{array}
        \right.
    \]
as functions of $r$.
\end{theorem}

\section{The Main Theorem}\label{sect:mainthm}

The structure of the proof of Theorem \ref{thm:main} is as follows. We begin with Lemma \ref{lma:pump1} by proving the existence of some $k'$ having properties which are necessary for subsequent constructions and such that $k | k'$. Without a loss of generality we will therefore assume that $k$ meets the conditions of the conclusion of Lemma \ref{lma:pump1}. In Lemma \ref{lma:rootunity} we establish a multivariate polynomial identity over $F$ which, after a suitable sequence of substitutions, gives a univariate identity with a degree $1$ polynomial, some terms of degree $k$, and some terms of intermediate degrees. The properties assumed on $k$ allow for a linear combination of two of these identities to eliminate all terms of intermediate degrees.

\begin{lemma}\label{lma:pump1}
Let $q$ and $k$ be relatively prime natural numbers and
\[
    k = k_1 + k_2 q^{a_2}+\cdots+k_N q^{a_N}
\]
be the base-$q$ expansion of $k$. Then there exists $k'$, $u \in\mathbb{N}$ satisfying the following conditions: $u > 1$, $k | k'$, and the base-$q$ expansion of $k'$ is given by
\[
    k' = k_1 + k_2 q^{a_2'}+\cdots+k_N q^{a_N'},
\]
where $a_i' \equiv 1 \pmod{u}$ for all $i\in\{2, \ldots, N\}$.
\end{lemma}

\begin{proof}
Denote by $r$ the order of $q$ in $(\mathbb{Z}/k\mathbb{Z})^\times$ and let $u > 1$ be relatively prime to $r$. For each $i\in\{2,\ldots,N\}$ there are infinitely many integers $b_i$ satisfying ${r b_i + a_i \equiv 1 \pmod{u}}$. Select the $b_i$ such that the sequence $a_i' = r b_i + a_i$ is positive and strictly increasing. Let $k'$ be given as in the statement of the lemma in terms of these $a_i'$. Since $q^r \equiv 1 \pmod{k}$ we have $q^{a_i'} \equiv q^{rb_i}q^{a_i} \equiv q^{a_i}\pmod{k}$. From which it follows that $k'\equiv k\pmod{k}$ and therefore $k|k'$.
\end{proof}

\begin{remark}\label{rem:pumping} Note that by construction we have $\gamma_q(k) = \gamma_q(k')$. In particular, $k'$ can be used in place of $k$ without changing the hypothesis or bound of Theorem \ref{thm:main}.
\end{remark}

The following lemma makes no assumption on the characteristic of $F$. Although in general we assume that $F$ is algebraically closed and of positive characteristic, Lemma \ref{lma:rootunity} only requires the existence of a suitable number of roots of unity.

\begin{lemma}\label{lma:rootunity}
Let $M\geq 2$ be an integer and $\omega \in F$ be a primitive $(M-1)$-th root of unity. Let $k_1,\ldots,k_N$ be natural numbers such that $\sum_{i=1}^N k_{i} = M$. Then the following identity holds in the polynomial ring $F[x_1,\ldots, x_N]$.
\begin{equation}\label{eq:multident}
\sum_{j=1}^{M-1} \prod_{i=1}^N(\omega^{j}+x_{i})^{k_i}=
    (M-1)\prod_{i=1}^{N}{x_{i}^{k_i}} +
    (M-1)\sum_{i=1}^{N}{k_ix_i}+\mathbf{1}_{M}(2).
\end{equation}
\end{lemma}

\begin{proof}
Observe that the coefficient of $\prod_{i=1}^N x_{i}^{d_i}$ in the left hand side of Equation \ref{eq:multident} is
	\[
		\left(\prod_{i=1}^{N}\binom{k_i}{d_i}\right)\sum_{j=1}^{M-1}(\omega^j)^{M-D},
	\]
where $D = \sum_{i=1}^N d_i$. We consider three cases in establishing the value of these coefficients.

\begin{enumerate}
\item[Case 1.] Suppose that $D = M$. Then $d_i = k_i$ for all $i$. The monomial is $\prod_{i=1}^N x_i^{k_i}$ and the coefficient is
    \[
		\left(\prod_{i=1}^{N}{\binom{k_i}{k_i}}\right)\sum_{j=1}^{M-1}(\omega^j)^0 = M-1.
    \]

\item[Case 2.] Suppose that $D=1$. Then exactly one $d_i$ is equal to $1$ and the rest are $0$. In which case, the monomial is $x_i$ and the coefficient is
	\[
		k_i\sum_{j=1}^{M-1}(\omega^j)^{M-1} = k_i(M-1).
	\]

\item[Case 3.] Suppose that $D \notin\{1, M\}$. If $M = 2$, then $D=0$ and $\omega = 1$. This corresponds to the constant term in the expansion of
\[
    \prod_{i=1}^N(1+x_{i})^{k_i},
\]
which is $1$.

On the other hand, if $M > 2$, then $M-1\nmid M-D$. Since $\omega$ is a primitive $(M-1)$-th root of unity it follows that $\omega^{M-D} \neq 1$. In particular,
    \[
        \sum_{j=1}^{M-1} (\omega^j)^{M-D} = \frac{\left(\omega^{M-D}\right)^{M-1}-1}{\left(\omega^{M-D}\right)-1} = 0.
    \]

Therefore, the contribution of these terms is $1$ if $M=2$ and $0$ otherwise.
\end{enumerate}

Combining these three cases gives the desired identity.
\end{proof}

We will assume without further mention that $k$ meets the conditions of the conclusion of Lemma \ref{lma:rootunity}. See Remark \ref{rem:pumping} for details.

\begin{proof}[Proof of Theorem \ref{thm:main}]
Let $q = p^n$ be as in the hypothesis and $\omega\in F$  be a primitive $(M-1)$-th root of unity. Let $k = k_1+k_2q^{a_2}+\cdots+k_N q^{a_N}$ be the base-$q$ expansion of $k$. The first step is to replace all of the multivariate products in Equation \ref{eq:multident} with univariate polynomials raised to the $k$-th power. Since $M-1 | q-1$ and $q-1 | q^{a_i}-1$, it follows that $(\omega^j)^{q^{a_i}-1} = 1$, so $(\omega^j)^{q^{a_i}} = \omega^j$ for all integers $j$ and all $i\in\{2, \ldots, N\}$. Therefore, working in $F[t]$, we have
\begin{equation}\label{eq:frob}
\prod_{i=1}^N\left(\omega^j+t^{q^{a_i}}\right)^{k_i} = \prod_{i=1}^N\left((\omega^j)^{q^{a_i}}+t^{q^{a_i}}\right)^{k_i}=
\left(\omega^j+t\right)^k,
\end{equation}
where the rightmost equality follows from repeated application of the Frobenius endomorphism.

Suppose that $\gamma_q(k) = M$, so that the hypothesis of Lemma \ref{lma:rootunity} is met. Substitute $x_i = t^{q^{a_i}}$ for $i\in\{2, \ldots, N\}$ and $x_1 = t$ into Equation \ref{eq:multident} and let the $k_i$ come from the base-$q$ expansion of $k$. After applying Equation \ref{eq:frob} and separating out the linear term, we get
\begin{equation}\label{eq:firstsubst}
\sum_{j=1}^{M-1} (\omega^j+t)^k = (M-1)t^k+(M-1)k_1t+(M-1)\sum_{i=2}^N k_i t^{q^{a_i}}+\mathbf{1}_M(2).
\end{equation}
It suffices to eliminate the last summation. Recall that $u$ is such that $a_i \equiv 1 \pmod{u}$ for $i\in\{2,\ldots,N\}$. Let $\lambda$ be a generator of the group $\mathbb{F}_{q^u}^\times\subset F$. For each $i\in\{2,\ldots, N\}$ write $a_i = g_i u + 1$ for some nonnegative integer $g_i$. Observe that
$q^{a_i}-q = q(q^{g_i u}-1)$ from which it follows $q^u-1 \mid q^{a_i}-q$. Since $\lambda$ has order $q^u-1$ we have $\lambda^{q^{a_i}-q} = 1$ and therefore $\lambda^{q^{a_i}} = \lambda^q$. In order to obtain another equation to work with we replace $t$ with $\lambda t$ in Equation \ref{eq:firstsubst} to get
\begin{multline}\label{eq:secondsubst}
\sum_{j=1}^{M-1} (\omega^j+\lambda t)^k = (M-1)\lambda^kt^k+(M-1)k_1\lambda t\\+(M-1)\lambda^q\sum_{i=2}^Nk_i t^{q^{a_i}}+\mathbf{1}_M(2).
\end{multline}
Next we subtract a $\lambda^q$-multiple of Equation \ref{eq:firstsubst} from Equation \ref{eq:secondsubst} to eliminate all terms of intermediate degree.
\begin{multline*}
\sum_{j=1}^{M-1} (\omega^j+\lambda t)^k - \sum_{j=1}^{M-1} \lambda^q(\omega^j+t)^k = (M-1)(\lambda^k-\lambda^q)t^k\\+(M-1)k_1(\lambda-\lambda^q)t+(1-\lambda^q)\mathbf{1}_M(2).
\end{multline*}
Since $M-1 \mid q-1$ it follows that $p \nmid M-1$ and therefore $M-1\neq 0$ in $F$. Similarly, because $\lambda$ has order $q^u-1$, with $u > 1$, we have $\lambda-\lambda^q \neq 0$. Finally, because $k$ is relatively prime to $p$, $k_1 \neq 0$, so the coefficient of $t$ is nonzero. After rearranging terms we get a degree one polynomial in $t$ equal to a sum of $2(M-1)+1$ perfect $k$-th powers. Therefore, $v(p,k) \leq 2(M-1)+1$ when $\gamma_q(k) = M$.

Suppose now that $\gamma_q(k) < M$ and let $R = M - \gamma_q(k)$. We can apply Lemma \ref{lma:rootunity} with exponents $k_1,k_2,\ldots,k_N, R$ to get
\begin{multline*}
\sum_{j=1}^{M-1} (\omega^j+x_{N+1})^R\prod_{i=1}^N(\omega^{j}+x_{i})^{k_i}=
    (M-1)x_{N+1}^R\prod_{i=1}^{N}{x_{i}^{k_i}}\\
    +(M-1)Rx_{N+1}+(M-1)\sum_{i=1}^{N}{k_ix_i}+\mathbf{1}_{M}(2)
\end{multline*}
which, after setting $x_{N+1} = 0$, simplifies to
\[
\sum_{j=1}^{M-1} \omega^{jR}\prod_{i=1}^N(\omega^{j}+x_{i})^{k_i}=(M-1)\sum_{i=1}^{N}k_i x_i+\mathbf{1}_{M}(2).
\]
From here the proof proceeds in the same way as the $\gamma_q(k) = M$ case but without the $t^k$ term. Therefore, we get one fewer $k$-th power in the final equation, so $v(p,k) \leq 2(M-1)$ when $\gamma_q(k) < M$.
\end{proof}

It should be noted that Equation \ref{eq:firstsubst} can also be established as a consequence of Lucas' theorem. To do so requires an analysis of the base-$q$ digits of numbers congruent to $1$ modulo $M-1$ and the breakdown of these digits into their base-$p$ expansion. The authors have instead opted for the proof given above as it is self-contained and provides their original insight into the problem.

\section{Applications of the Main Theorem}\label{section.applications}
The primary weakness of Theorem \ref{thm:main} is the requirement that the sum of the digits be at most the base we are working over.  The following lemma allows us to systematically increase the base to get some $k'$, divisible by $k$, without changing the sum of the digits. This pumping of $k$ is the main tool used to prove the remaining theorems.

\begin{lemma}\label{lma:pump2}
Let $q$, $k \in\mathbb{N}$ be relatively prime and $q$ have order $r$ in $(\mathbb{Z}/k\mathbb{Z})^\times$. For any $u\in\mathbb{N}$, relatively prime to $r$, there exists some natural number $k'$ such that $k\mid k'$ and $\gamma_{q^u}(k^\prime) = \gamma_q(k)$.
\end{lemma}
\begin{proof}
    Since $u$ and $r$ are relatively prime there exists some $b\in\mathbb{N}$ such that $rb\equiv -1 \pmod{u}$. By our hypothesis, $q^r \equiv 1\pmod{k}$, it follows that $q^{rb+1}\equiv q\pmod{k}$. Let $k_1+k_2q^{a_2}+\cdots+k_N q^{a_N}$ be the base-$q$ expansion of $k$ and define $k'$ by replacing $q$ with $q^{rb+1}$ in this expansion. Then we have
    \[
        k' = k_1+k_2 q^{(rb+1)a_2}+\cdots+k_N q^{(rb+1)a_N}\equiv k \pmod{k}.
    \]
    However, $u \mid (rb+1)$, so the nonzero digits in the base-$q^u$ expansion of $k'$ are the same as the nonzero digits in the base-$q$ expansion of $k$. Therefore, $\gamma_{q^u}(k^\prime) = \gamma_q(k)$.
\end{proof}

With Lemma \ref{lma:pump2} in place we now proceed to prove Theorem \ref{thm:odd} and Theorem \ref{thm:order}.

\begin{proof}[Proof of Theorem \ref{thm:odd}]
Let $r$ be the order of $p$ in $(\mathbb{Z}/k\mathbb{Z})^\times$ and $c$ be the least nonnegative integer such that $\gamma_p(k) \leq 2^c+1$. In the notation of Lemma \ref{lma:pump2} we let $u = 2^c$, in which case there exists some $k'$, divisible by $k$, satisfying $\gamma_{p^u}(k') = \gamma_{p}(k)$. Since $p$ is odd it has order dividing $2^c$ in $(\mathbb{Z}/2^{c+1}\mathbb{Z})^\times$, and thus $2^{c+1} \mid p^u-1$, so $2^c\mid p^u-1$. Let $M = 2^c+1$ and $n = u$ in Theorem \ref{thm:main} to get $v(p, k') \leq 2^{c+1}+\mathbf{1}_{2^c+1}(\gamma_{p^u}(k'))$. Since $v(p, k) \leq v(p, k')$ and $\gamma_{p^u}(k') = \gamma_{p}(k)$ we have
\[
    v(p, k) \leq 2^{c+1}+\mathbf{1}_{2^c+1}(\gamma_{p}(k)).
\]
For the second inequality note that $2^{c-1} < \gamma_p(k)$, so $2^{c+1} < 4\gamma_p(k)$. Since all quantities are integers it follows that $2^{c+1}+1\leq 4\gamma_p(k)$. However, the left hand side is odd and the right hand side is even, therefore the inequality strictly holds.
\end{proof}

\begin{proof}[Proof of Theorem \ref{thm:order}]
Let $u$ be as in the statement of the theorem. By Lemma \ref{lma:pump2} there exists some $k'$, divisible by $k$, such that $\gamma_{p^u}(k') = \gamma_p(k)\leq p^u$. Let $M = p^u$ and $n=u$ in Theorem \ref{thm:main} to get $v(p, k') \leq 2(p^u-1)+\mathbf{1}_{p^u}(\gamma_{p^u}(k'))$. Since $v(p, k) \leq v(p, k')$ and $\gamma_{p^u}(k') = \gamma_{p}(k)$ the result immediately follows.
\end{proof}

Before beginning the proof of Theorem \ref{thm:asymptotic} we recall the definition and some known bounds on the first Chebyshev function. For any real number $x$ define
\[
\vartheta(x) = \sum_{\substack{\ell \leq x,\\ \ell\text{ is prime}}} \ln(\ell).
\]
Theorem 4 of Rosser and Schoenfeld \cite{barkley1962approximate} establishes the bound
\begin{equation}\label{eq:chebybound}
x\left(1-\frac{1}{2\ln(x)}\right) < \vartheta(x) < x\left(1+\frac{1}{2\ln(x)}\right)
\end{equation}
for all $x \geq 563$. We will use these bounds to find small primes not dividing $r$.
\begin{proof}[Proof of Theorem \ref{thm:asymptotic}]
We first consider the case when $p > 2$. Let $\epsilon > 0$ be arbitrary. As a consequence of Equation \ref{eq:chebybound} we have
\[
x\left(1+\epsilon - \frac{2+\epsilon}{2\ln((2+\epsilon)x)} - \frac{1}{2\ln(x)}\right) < \vartheta((2+\epsilon)x) - \vartheta(x)
\]
for all $x\geq 563$. Since
\[
\lim_{x\rightarrow\infty}\left(\frac{2+\epsilon}{2\ln((2+\epsilon)x)} + \frac{1}{2\ln(x)}\right) = 0
\]
there exists some $N$, depending on $\epsilon$, such that for all $x > N$ the inequality
\[
\vartheta((2+\epsilon)x) - \vartheta(x) > x
\]
holds. If $r$ is large enough, so that $\ln(r) > N$, then
\begin{equation}\label{eq:partialsum}
\sum_{\substack{\ln(r) < \ell \leq (2+\epsilon)\ln(r),\\ \ell\text{ is prime}}} \ln(\ell) > \ln(r).
\end{equation}
Exponentiating both sides of the above inequality gives
\[
\prod_{\substack{\ln(r) < \ell \leq (2+\epsilon)\ln(r),\\ \ell\text{ is prime}}} \ell > r.
\]
Therefore, $r$ cannot be divisible by all primes between $\ln(r)$ and $(2+\epsilon)\ln(r)$. Let $\ell$ be a prime in this range such that $\ell \nmid r$. By construction $p^\ell \geq p^{\ln(r)} = r^{\ln(p)}$. Since $\ln(p) > 1$, for $r$ sufficiently large we have $r^{\ln(p)} > r(p-1) = \gamma_p(p^r-1)$. Apply Lemma \ref{lma:pump2} with $u = \ell$ to get some $k'$ satisfying $p^r-1 | k'$ and $\gamma_{p^\ell}(k') = r(p-1)$. Consequently, we have $\gamma_{p^\ell}(k') \leq p^\ell$. In the notation of Theorem \ref{thm:main}, let $M = p^\ell$ and $n=\ell$ to get
\[
v(p, k') \leq 2(p^\ell-1)+\mathbf{1}_{p^\ell}(\gamma_{p^\ell}(k')) < 2p^\ell.
\]

On the other hand $\ell \leq (2+\epsilon)\ln(r)$, therefore, $p^\ell \leq r^{(2+\epsilon)\ln(p)}$. Since $p^r-1\mid k'$ it follows that
\[
v(p, p^r-1) \leq v(p, k') < 2r^{(2+\epsilon)\ln(p)}.
\]
In particular, $v(p, p^r-1) = \mathcal{O}(r^{(2+\epsilon)\ln(p)})$ as a function of $r$.

In the case that $p=2$, the above argument only breaks down because
\[
r^{\ln(2)} < r = \gamma_2(2^r-1).
\]
To remedy this we use the inequality
\[
\sum_{\substack{\log_2(r) < \ell \leq (2+\epsilon)\log_2(r),\\ \ell\text{ is prime}}} \ln(\ell) > \log_2(r)
\]
instead of the one presented in Equation \ref{eq:partialsum}. The rest of the proof is the same with the exception of the bound on $2^\ell$ which becomes
\[
2^{(2+\epsilon)\log_2(r)} = r^{2+\epsilon}.
\]
This in turn gives $v(2, 2^r-1) < 2r^{2+\epsilon}$ for all sufficiently large $r$.
\end{proof}

\bibliographystyle{alpha}
\bibliography{bibliography}
\end{document}